\documentclass[11pt]{article}
\usepackage{amsmath,amsfonts,amssymb,amsthm,anysize}
\usepackage{amsmath}
\usepackage{amssymb}
\usepackage{amsbsy}
\usepackage{amsfonts}

\theoremstyle{plain}
\newtheorem{theorem}{Theorem}[section]

\newtheorem{prop}[theorem]{Proposition}

\theoremstyle{definition}

\newtheorem{ex}[theorem]{Example}

\begin{document}

\title{Cohomological Property of Vector Bundles on Biprojective
Spaces}
\author{Francesco Malaspina and Chikashi Miyazaki
\vspace{4pt}\\
{\small Politecnico di Torino}\\
{\small\it Corso Duca degli Abruzzi 24, 10129 Torino, Italy}\\
{\small\it e-mail: francesco.malaspina@polito.it}\\
\vspace{6pt}\\
{\small  Department of Mathematics,
Faculty of Education, Kumamoto University}\\
{\small\it  Kurokami 2-40-1, Chuo-ku,
Kumamoto 860-8555, Japan}\\
{\small\it e-mail: cmiyazak@educ.kumamoto-u.ac.jp}}
\date{}

\maketitle

\def\thefootnote{}
\footnote{The first  author is member of the GNSAGA group of INdAM. The second author was partially supported by
INdAM(Italy) and JSPS Kakenhi(C) (26400048). \par
}

\begin{abstract}
This paper investigates the cohomological property
of vector bundles on biprojective space. We will give
a criterion for a vector bundle to be isomorphic to
the tensor product of pullbacks of exterior products of
differential sheaves.
\end{abstract}

\section{Introduction}

The purpose of this paper is to study the cohomological
property of vector bundle towards
Horrocks-type criteria and to characterize the tensor product of
pullbacks of exterior products of
differential sheaves on biprojective space.
Horrocks Theorem says that an ACM vector bundle on the
projective space is isomorphic to a direct sum of
line bundles. There are some attempts to generalize
to the biprojective space, that is,
some splitting criteria for a vector bundle on
${\Bbb P}^{m} \times {\Bbb P}^{n}$ to be isomorphic
to a direct sum of the form
${\cal O}_{{\Bbb P}^{m} \times {\Bbb P}^{n}}(\ell_1, \ell_2)$
in \cite{BM,CM,M1}. In particular in \cite{CM} it is used a
Beilinson type spectral sequence  and $m$-blocks collection
while in \cite{BM, M1} it is
used a notion of Castelnuovo-Mumford regularity and Koszul complexes (for a similar approach on Grassmannians see \cite{AM}).
In this paper we will give a cohomological criterion
for a vector bundle on ${\Bbb P}^{m} \times {\Bbb P}^{n}$
to have a direct summand of the form
${\Omega}^p_{{\Bbb P}^{m}}
\boxtimes {\Omega}^q_{{\Bbb P}^{n}}$ using the second approach.

Let us describe our perspective on the
condition for a vector bundle to have a specific
direct summand. We will begin a proof of the Horrocks
theorem through the
Castelnuovo-Mumford regularity according to
\cite{BM}. Let $E$ be an ACM vector bundle on ${\Bbb P}^n$.
Assume that $E$ is $m$-regular but not $(m-1)$-regular,
see \cite{Mum} for the definition and basic properties
for $m$-regular.
Then we have a surjective map
${\displaystyle \varphi :
{\cal O}_{{\Bbb P}^n}^{\oplus} \to E(m)}$.
Since $E$ is ACM, we have ${\rm H}^n(E(m-1-n)) \ne 0$,
and ${\rm H}^0(E^{\vee}(-m)) \ne 0$ by Serre duality.
Thus we have a nonzero map ${\displaystyle
\psi : E(m) \to {\cal O}_{{\Bbb P}^n}}$.
Since $\psi \circ \varphi$ is nonzero,
it splits. Hence ${\cal O}_{{\Bbb P}^n}$ is
a direct summand of $E(m)$.

Now we will proceed the next step on studying
a Buchsbaum vector bundle on ${\Bbb P}^n$,
that is, ${\frak m}{\rm H}^i_*(E|_L) = 0$
for any $r$-plane $L$ of ${\Bbb P}^n$, $1 \le i < r \le n$,
where ${\frak m} = \oplus_{\ell \ge 1}\Gamma({\cal O}_{{\Bbb P}^n}(\ell))$.
Instead of using the regularity we will make use of
the Koszul complex.
Before giving a sufficient
condition for a vector bundle to be a direct sum of vector bundles
of the form  $\Omega^i_{{\Bbb P}^n}(\ell)$,
we will describe important facts concerning the structure
of Buchsbaum vector bundle.

\begin{prop}[\cite{C,G}]
\label{Goto-Chang}
Let $E$  be a Buchsbaum vector bundle on ${\Bbb P}^n$.
Then $E$ is isomorphic to a direct sum of vector bundles
of the form $\Omega^i_{{\Bbb P}^n}(\ell)$.
\end{prop}

\begin{prop}[\cite{SV}, (I.3.10)]
\label{SV}
Let $E$  be a vector bundle on ${\Bbb P}^n$.
Let us define ${\frak S} = \{ (i,\ell) | 1 \le i \le n-1, \ell \in {\Bbb Z},
{\rm H}^i(E(\ell)) \ne 0 \}$. Suppose that ${\frak S}$
satisfies the following condition:
``For $(i,\ell), (j,m) \in {\frak S}$, if $i \ge j$, then $i + \ell + 1 \ne
j + m$''.
Then $E$ is Buchsbaum.
\end{prop}

From these facts we have observed the relation between the Buchsbaum
property and the differential sheaves in (\ref{Goto-Chang})
and a cohomological characterization of Buchsbaum modules in
(\ref{SV}). Then we will give a straightforward proof of
a more or less known result which illustrates the relation
between the vanishings of the intermediate cohomologies
and the exterior products of differential sheaves. The following
is a starting point of our main result, Theorem~\ref{mainth}.

\begin{prop}
\label{start}
Let $E$  be a vector bundle on ${\Bbb P}^n$
with ${\rm H}^p(E) \ne 0$, where $1 \le p \le n-1$.
If a vector bundle $E$ has the following condition:
\begin{itemize}
\item[{\rm (a)}]
${\rm H}^i(E(p-i+1)) = 0$ for $1 \le i \le p$.
\item[{\rm (b)}]
${\rm H}^i(E(p-i-1)) = 0$ for $p \le i \le n - 1$,
\end{itemize}
\noindent
then $E$ contains
${\Omega}^p_{{\Bbb P}^n}$
as a direct summand.
\end{prop}

\begin{proof}
By an exact sequence arising from the Koszul complex:
\[ 0 \to {\cal O}_{{\Bbb P}^n} \to
 {\cal O}_{{\Bbb P}^n}^{\oplus}(1)
\to \cdots \to  {\cal O}_{{\Bbb P}^n}^{\oplus}(p) \to
\Omega_{{\Bbb P}^n}^{p \vee} \to 0, \]
we have a surjective map $\varphi: {\rm H}^0(E \otimes
\Omega^{p \vee}_{{\Bbb P}^n}) \to {\rm H}^p(E)$ from the
assumption
${\rm H}^1(E(p)) = \cdots = {\rm H}^p(E(1)) = 0$.
By an exact sequence arising from the Koszul complex:
\[ 0 \to {\cal O}_{{\Bbb P}^n}(-n-1) \to
 {\cal O}_{{\Bbb P}^n}^{\oplus}(-n)
\to \cdots \to  {\cal O}_{{\Bbb P}^n}^{\oplus}(-p-1) \to
\Omega_{{\Bbb P}^n}^p \to 0, \]
we have a surjective map $\psi : {\rm H}^0(E^{\vee} \otimes
\Omega^p_{{\Bbb P}^n}) \to {\rm H}^{n-p}(E^{\vee}(-n-1))$ from
the assumption
${\rm H}^p(E(-1)) = \cdots = {\rm H}^{n-1}(E(p-n)) = 0$,
that is, ${\rm H}^1(E^{\vee}(-p-1)) = \cdots =
{\rm H}^{n-p}(E^{\vee}(-n)) = 0$.
\par

As in \cite{BM} we have a nonzero element $f \in
{\rm H}^0(E \otimes
\Omega^{p \vee}_{{\Bbb P}^n})$ such that
$\varphi(f) = s (\ne 0) \in
{\rm H}^p(E)$.
Let us take
an element $s^{*} \in {\rm H}^{n-p}(E^{\vee}
(-n-1))$ corresponding to $s \in
{\rm H}^{m}(E)$,
 there is a nonzero element $g \in
{\rm H}^0(E^{\vee} \otimes
\Omega^p_{{\Bbb P}^n})$ such that
$\psi(g) = s^{*} (\ne 0) \in
{\rm H}^{n-p}(E^{\vee}(-n-1))$.
Then $f$ and $g$ are regarded as elements
of ${\rm Hom}(\Omega^p_{{\Bbb P}^n}, E)$ and
${\rm Hom}(E,\Omega^p_{{\Bbb P}^n})$ respectively.
From a commutative diagram:
\[
\begin{array}{ccc}
{\rm H}^0(E \otimes \Omega^{p \vee}_{{\Bbb P}^n})
\otimes {\rm H}^0(E^{\vee} \otimes \Omega^p_{{\Bbb P}^n})
& \to  &
{\rm H}^0(\Omega^{p \vee}_{{\Bbb P}^n} \otimes \Omega^p_{{\Bbb P}^n})  \cong
 {\rm H}^{0}({\cal O}_{{\Bbb P}^n})
\\
\downarrow & &  \downarrow   \\
{\rm H}^p(E) \otimes {\rm H}^{n-p}(E^{\vee}(-n-1))
  &
\to  &
 {\rm H}^n({\cal O}_{{\Bbb P}^n}(-n-1) ),
\\
\end{array}
\]
\noindent
a natural map
${\displaystyle {\rm H}^0(E \otimes \Omega^{p \vee}_{{\Bbb P}^n})
\otimes {\rm H}^0(E^{\vee} \otimes \Omega^p_{{\Bbb P}^n})
 \to
 {\rm H}^{0}({\cal O}_{{\Bbb P}^n})}$ yields that
$g \circ f$ is an isomorphism, which
implies  $\Omega^p_{{\Bbb P}^n}$ is a direct summand of $E$.
\end{proof}

\section{Cohomological Criterion of Vector Bundles on Biprojective Space}

What condition is required for a vector
bundle $E$ on ${\Bbb P}^m \times {\Bbb P}^n$
 to have a direct summand of the form
${\Omega}^p_{{\Bbb P}^{m}} \boxtimes  {\Omega}^q_{{\Bbb P}^{n}}$?
Although the exterior products of differential
sheaves are the indecomposable Buchsbaum vector bundles on ${\Bbb P}^n$,
The Buchsbaum property of differential sheaves are more complicated on
${\Bbb P}^m \times {\Bbb P}^n$, see \cite{M2}. This section is devoted
to an answer of cohomological criteria of differential sheaves from
the viewpoint of (\ref{start}).
Compared with an important result of \cite[(4.11)]{CM}, our theorem obtained
from an elementary way concludes an isomorphism to just one bundle directly.
\par

\begin{theorem}
\label{mainth}
Let $E$ be a vector bundle on ${\Bbb P}^{m} \times {\Bbb P}^{n}$
with ${\rm H}^{p+q}(E) \ne 0$, where $1 \le p \le m-1$ and
$1 \le q \le n-1$.
If a vector bundle $E$ has the following condition:
\begin{itemize}
\item[{\rm (a)}]
${\rm H}^i(E(a,b)) = 0$ for $1 \le i \le p+q$, $0 \le
a \le p$, $0 \le b \le q$ with $i + a + b = p + q + 1$.
\item[{\rm (b)}]
${\rm H}^i(E(a,b)) = 0$ for $p+q \le i \le m+n-1$,
$p - m \le a \le 0$, $q - n \le b \le 0$ with
$i + a + b = p + q - 1$,
\end{itemize}
\noindent
then $E$ contains
${\Omega}^p_{{\Bbb P}^{m}}
\boxtimes  {\Omega}^q_{{\Bbb P}^{n}}$
as a direct summand.
\end{theorem}

\begin{proof}
Let us consider the exact sequences arising from the Koszul complexes
\begin{equation}\label{e1}0 \to
\Omega_{{\Bbb P}^{m}}^p \to {\cal O}_{{\Bbb P}^{m}}^{f_p}(-p)
\to  {\cal O}_{{\Bbb P}^{m}}^{f_{p-1}}(-p+1) \to \cdots
\to  {\cal O}_{{\Bbb P}^{m}}^{f_1}(-1) \to {\cal O}_{{\Bbb P}^{m}}
\to 0
\end{equation}
and
\begin{equation}\label{e2}
0 \to \Omega_{{\Bbb P}^{n}}^q \to {\cal O}_{{\Bbb P}^{n}}^{e_q}(-q)
\to  {\cal O}_{{\Bbb P}^{n}}^{e_{q-1}}(-q+1) \to \cdots
\to  {\cal O}_{{\Bbb P}^{n}}^{e_1}(-1) \to {\cal O}_{{\Bbb P}^{n}}
\to 0,
\end{equation}
where ${\displaystyle f_i = \left(
\begin{array}{c}
m \\
i
\end{array}
\right)}$ and ${\displaystyle e_j = \left(
\begin{array}{c}
n \\
j
\end{array}
\right)}$. By gluing the pull back by
$p_2 : {\Bbb P}^m \times {\Bbb P}^n \to {\Bbb P}^n$
of the dual of (\ref{e2})  tensored by $E$ and the pull back by
$p_1 : {\Bbb P}^m \times {\Bbb P}^n \to {\Bbb P}^m$
 of the dual of (\ref{e1}) tensored by $E\otimes p_2^*
\Omega_{{\Bbb P}^{n}}^{q \vee}$
we obtain an exact sequence
\[ 0 \to E \to E(0,1)^{e_1} \to E(0,2)^{e_2} \to \cdots
\to E(0,q)^{e_q} \to E(1,0)^{f_1} \otimes p_2^*
\Omega_{{\Bbb P}^{n}}^{q \vee}\]
\[ \to E(2,0)^{f_2} \otimes p_2^* \Omega_{{\Bbb P}^{n}}^{q \vee}
\to \cdots \to E(p,0)^{f_p}
\otimes p_2^* \Omega_{{\Bbb P}^{n}}^{q \vee}
\to E \otimes
p_1^* \Omega_{{\Bbb P}^{m}}^{p \vee}
\otimes p_2^* \Omega_{{\Bbb P}^{n}}^{q \vee} \to 0. \]
Notice that \[\Omega_{{\Bbb P}^{n}}^{q \vee}.
\cong
\Omega_{{\Bbb P}^{n}}^{n-q }(n+1)\]
In order to have a surjective map
\[ \varphi: {\rm H}^0(E \otimes
p_1^* \Omega_{{\Bbb P}^{m}}^{p \vee}
\otimes p_2^* \Omega_{{\Bbb P}^{n}}^{q \vee})
 \to {\rm H}^{p+q}(E), \]
we will show
\begin{itemize}
\item[(c.1)]
${\rm H}^1(E(p,n+1) \otimes p_2^* \Omega_{{\Bbb P}^{n}}^{n-q})
= \cdots =
{\rm H}^p(E(1,n+1) \otimes p_2^* \Omega_{{\Bbb P}^{n}}^{n-q})
= 0$.
\item[(c.2)]
${\rm H}^{p+1}(E(0, q)) = \cdots = {\rm H}^{p+q}(E(0,1)) = 0$.
\end{itemize}
The asserion (c.2) follows from the assumption (a).
Since ${\rm H}^i(E(p-i+1,q)) = {\rm H}^{i+1}(E(p-i+1,q-1)) = \cdots
= {\rm H}^{i+q}(E(p-i+1, 0)) = 0$, $i = 1, \cdots, p$,
we see that ${\rm H}^i(E(p-i+1,n+1) \otimes
p_2^* \Omega_{{\Bbb P}^{n}}^{n-q}) = 0$
from the exact sequence
\[ 0 \to {\cal O}_{{\Bbb P}^{n}} \to
 {\cal O}_{{\Bbb P}^{n}}^{e_{n}}(1)
\to \cdots \to  {\cal O}_{{\Bbb P}^{n}}^{e_{n-q+1}}(q) \to
\Omega_{{\Bbb P}^{n}}^{n-q}(n+1) \to 0 \]
\noindent
by pulling back to $p_2$ and tensored by $E(p-i+1,0)$.
Thus we obtain (c.1).
\par
\medskip

Next, let us consider the exact sequences arising from the Koszul complexes
\begin{equation}\label{e3} 0 \to {\cal O}_{{\Bbb P}^{m}}(-m-1) \to
 {\cal O}_{{\Bbb P}^{m}}^{f_{m}}(-m)
\to \cdots \to  {\cal O}_{{\Bbb P}^{m}}^{f_{p+1}}(-p-1) \to
\Omega_{{\Bbb P}^{m}}^p \to 0 \end{equation}
and
\begin{equation}\label{e4} 0 \to {\cal O}_{{\Bbb P}^{n}}(-n-1) \to
 {\cal O}_{{\Bbb P}^{n}}^{e_{n}}(-n)
\to \cdots \to  {\cal O}_{{\Bbb P}^{n}}^{e_{q+1}}(-q-1) \to
\Omega_{{\Bbb P}^{n}}^q \to 0 \end{equation}
By gluing the pull back by $p_2$ of  (\ref{e4})
tensored by $E^\vee(-m-1,0)$ and the pull back by $p_1$
of  (\ref{e3}) tensored by $E^{\vee}\otimes p_2^*
\Omega_{{\Bbb P}^{n}}^{q }$
we obtain an exact sequence
\[ 0 \to E^{\vee}(-m-1,-n-1) \to  E^{\vee }(-m-1,-n)^{e_{m}}
\to \cdots\] \[\cdots \to  E^{\vee }(-m-1,-q-1)^{e_{p+1}}
\to  E^{\vee }(-m,0)^{f_{m}} \otimes p_2^* \Omega_{{\Bbb P}^{n}}^q
\to \cdots\] \[ \cdots\to  E^{\vee }(-p-1,0)^{f_{p+1}}
\otimes p_2^* \Omega_{{\Bbb P}^{n}}^q
\to E^{\vee} \otimes
p_1^* \Omega_{{\Bbb P}^{m}}^p \otimes p_2^* \Omega_{{\Bbb P}^{n}}^q \to 0.
\]
In order to have a surjective map
\[ \psi:  {\rm H}^0(E^{\vee} \otimes p_1^* \Omega_{{\Bbb P}^{m}}^p
\otimes p_2^* \Omega_{{\Bbb P}^{n}}^q)
 \to {\rm H}^{m+n-p-q}(E^{\vee}(-m-1,-n-1)) \]
we will show that
${\rm H}^1(E^{\vee}(-p-1,0) \otimes p_2^* \Omega_{{\Bbb P}^{n}}^q)
= \cdots
= {\rm H}^{m - p}(E^{\vee}(-m,0) \otimes p_2^* \Omega_{{\Bbb P}^{n}}^q) = 0$
and
${\rm H}^{m-p+1}(E^{\vee}(-m-1,-q-1)) = \cdots =
{\rm H}^{m+n-p-q}(E^{\vee}(-m-1,-n) = 0$.
By Serre duality, we have only to show
\begin{itemize}
\item[(d.1)]
${\rm H}^{m+n-1}(E(p-m,0) \otimes p_2^* \Omega_{{\Bbb P}^{n}}^{n-q})
= \cdots
= {\rm H}^{n + p}(E(-1,0) \otimes p_2^* \Omega_{{\Bbb P}^{n}}^{n-q}) = 0$.
\item[(d.2)]
${\rm H}^{n+p-1}(E(0,-n+q)) = \cdots =
{\rm H}^{p+q}(E(0,-1)) = 0$.
\end{itemize}
The assertion (d.2) follows from the assumption (b).
Since ${\rm H}^i(E(n+p-i-1,-n+q)) = \cdots
= {\rm H}^{i-n+q}(E(n+p-i-1,0)) = 0$, $i = n+p, \cdots, m+n-1$,
from the assumption $(b)$,
we see that ${\rm H}^i(E(n+p-i-1,0) \otimes
 p_2^* \Omega_{{\Bbb P}^{n}}^{n-q}) = 0$
from the exact sequence
\[ 0 \to \Omega_{{\Bbb P}^{n}}^{n-q} \to
{\cal O}_{{\Bbb P}^{n}}^{e_{n-q}}(-n+q) \to \cdots \to
 {\cal O}_{{\Bbb P}^{n}}^{e_1}(-1)
\to  {\cal O}_{{\Bbb P}^{n}}
\to 0 \]
\noindent
by pulling back to $p_2$ and tensored by $E(n+p-i-1,0)$.
Thus we obtain (d.1).
\par
\medskip

As in (\ref{start}), for a nonzero element $s \in  {\rm H}^{p+q}(E)$
and the corresponding element $s^* \in {\rm H}^{m+n-p-q}(E^{\vee}(-m-1,-n-1))$
by Serre duality, there are $f \in {\rm H}^0(E \otimes
p_1^* \Omega_{{\Bbb P}^{m}}^{p \vee}
\otimes p_2^* \Omega_{{\Bbb P}^{n}}^{q \vee})$
and $g \in {\rm H}^0(E^{\vee} \otimes p_1^*
\Omega_{{\Bbb P}^{m}}^p
\otimes p_2^* \Omega_{{\Bbb P}^{n}}^q)$
with $\varphi(f) = s$ and $\psi(g) = s^*$.
Then $g \circ f$ is an isomorphism by regarding
as $f \in {\rm Hom}(p_1^* \Omega_{{\Bbb P}^{m}}^p
\otimes p_2^* \Omega_{{\Bbb P}^{n}}^q, E)$
and $g \in {\rm Hom}(E,
p_1^* \Omega_{{\Bbb P}^{m}}^p
\otimes p_2^* \Omega_{{\Bbb P}^{n}}^q)$, which
gives an inclusion from
$p_1^* {\Omega}^p_{{\Bbb P}^{m}}
\otimes p_2^* {\Omega}^q_{{\Bbb P}^{n}}$
to $E$ as a direct summand.
\end{proof}

\begin{ex}
\label{P2*P2}
Let us give an application of (\ref{mainth})
 to a vector bundle of ${\Bbb P}^2 \times
{\Bbb P}^2$.
Let $E$ be an indecomposable vector bundle on
${\Bbb P}^2 \times {\Bbb P}^2$. Then the following
conditions are equivalent: \par
\begin{enumerate}
\item[(a)] $E \cong \Omega_{{\Bbb P}^2} \boxtimes
\Omega_{{\Bbb P}^2}$.
\item[(b)]
${\rm H}^2(E) \ne 0$ and
${\rm H}^1(E(1,1)) = {\rm H}^2(E(0,1)) = {\rm H}^2(E(1,0))
= {\rm H}^2(E(-1,0)) = {\rm H}^2(E(0,-1)) = {\rm H}^3(E(-1,-1))
= 0$.
\end{enumerate}

\end{ex}


\begin{thebibliography}{99}

\bibitem{AM} E.~Arrondo and F.~Malaspina, Cohomological
  characterization of vector bundles on Grassmannians of lines,
  J. Algebra 323 (2010), no.~4, 1098--1106.


\bibitem{BM}
E.~Ballico and F.~Malaspina, Regularity and cohomological splitting
conditions for vector bundles on multiprojective spaces,
J.~Algebra 345 (2011), 137 -- 149.

\bibitem{C}
M.~C.~Chang, Characterization of arithmetically Buchsbaum subschemes
of codimension 2 in ${\Bbb P}^n$,
J. Differential Geom. 31 (1990), 323--341.

\bibitem{CM}
L.~Costa and R.~M.~Mir\'{o}-Roig,
Cohomological characterization of vector bundles on
multiprojective spaces,  J.~Algebra  294  (2005),
73--96, with a corrigendum in J.~Algebra 319  (2008),
1336--1338.

\bibitem{G}
S.~Goto,  Maximal Buchsbaum modules over regular local rings
and a structure theorem for generalized Cohen-Macaulay modules,
ASPM 11(1987), 39--64.

\bibitem{M1}
C.~Miyazaki, A cohomological criterion for splitting
of vector bundles on multiprojective space, Proc.~Amer.~Math.~Soc.
143 (2015), 1435--1440.

\bibitem{M2}
C.~Miyazaki, Buchsbaum criterion of Segre products of
vector bundles on multiprojective space, J.~Algebra 467
(2016), 47--57.

\bibitem{Mum}
D. Mumford, Lectures on curves on an algebraic surface,
Annals of Math.~Studies 59 (1966), Princeton UP.

\bibitem{SV} J.~St\"uckrad and W.~Vogel, Buchsbaum rings and applications,
Springer, 1986.


\end{thebibliography}
\end{document}